\documentclass[sn-mathphys,Numbered]{sn-jnl}

\usepackage{graphicx}%
\usepackage{multirow}%
\usepackage{amsmath,amssymb,amsfonts}%
\usepackage{amsthm}%
\usepackage{mathrsfs}%
\usepackage[title]{appendix}%
\usepackage{xcolor}%
\usepackage{textcomp}%
\usepackage{manyfoot}%
\usepackage{booktabs}%
\usepackage{algorithm}%
\usepackage{algorithmicx}%
\usepackage{algpseudocode}%
\usepackage{listings}%

\newtheorem{theorem}{Theorem}[section]
\newtheorem{lemma}[theorem]{Lemma}

\newtheorem{example}[theorem]{Example}

\def\stab{{\rm Stab}}
\def\aut{{\rm Aut}}

\begin{document}

\title[Latin bitrades derived from quasigroup autoparatopisms]{Latin bitrades derived from quasigroup autoparatopisms}


\author[1]{\fnm{Nicholas J.} \sur{Cavenagh}}\email{nickc@waikato.ac.nz}

\author*[2]{\fnm{Ra\'ul M.} \sur{Falc\'on}}\email{rafalgan@us.es}
\equalcont{These authors contributed equally to this work.}

\affil*[1]{\orgdiv{Department of Mathematics}, \orgname{The University of Waikato}, \orgaddress{\street{Private Bag 3105}, \city{Hamilton}, \postcode{3240}, \country{New Zealand}}}

\affil[2]{\orgdiv{Department of Applied Mathematics I}, \orgname{Universidad de Sevilla}, \orgaddress{\street{Avenida Reina Mercedes 4 A}, \city{Sevilla}, \postcode{41012}, \country{Spain}}}

\abstract{In 2008, Cavenagh and Dr\'{a}pal, et al, described a method of constructing Latin trades using groups. 
The Latin trades that arise from this construction are entry-transitive (that is, there always exists an autoparatopism of the Latin trade mapping any ordered triple to any other ordered triple). Moreover, useful properties of the Latin trade can be established using properties of the group. However, the construction does not give a direct embedding of the Latin trade into any particular Latin square. In this paper, we generalize the above to construct Latin trades embedded in a Latin square $L$, via the autoparatopism group of the quasigroup with Cayley table $L$. We apply this theory to identify non-trivial entry-transitive trades in some group operation tables as well as in Latin squares that arise from quadratic orthomorphisms.}



\maketitle

\section{Introduction}
\label{Section:Intro}

A {\em partial Latin square} of order $n$ is an $n\times n$ array $P=(P[i,j])$, whose cells contain symbols in a finite set $Q$ of $n$ distinct symbols, so that each symbol appears at most once per row, and at most once per column. Those cells containing no symbol are called {\em empty}.  The number of non-empty cells in $P$ is its {\em size}, which we denote by $|P|$. If $|P|=n^2$, then the array is a {\em Latin square} of order $n$. Furthermore, the partial Latin square $P$ is uniquely identified with its {\em set of entries} $\mathrm{Ent}(P):=\left\{(i,j,P[i,j])\mid i,j,P[i,j]\in Q\right\}$. 
Observe that any two distinct entries of a partial Latin square agree in at most one coordinate. 
If $\mathrm{Ent}(P')\subseteq \mathrm{Ent}(P)$, for some other partial Latin square $P'$ of the same order, then it is said that $P'$ is {\em embedded} in $P$. This is denoted by $P'\subseteq P$.

A {\em Latin trade} in a Latin square $L$ is any non-empty partial Latin square $T\subseteq L$ such that there is another partial Latin square $T'$ of the same order satisfying the following three conditions: (1) $|T|=|T'|$; (2)  $\mathrm{Ent}(T)\cap\mathrm{Ent}(T')=\emptyset$; and (3) the set $\left(\mathrm{Ent}(L)\setminus \mathrm{Ent}(T)\right)\cup \mathrm{Ent}(T')$ is the set of entries of a new Latin square. 
Condition (3) above is equivalent to the following: ($3^{\ast}$) corresponding rows and columns of $T$ and $T'$ contain the same set of symbols. 
Conditions (1), (2) and ($3^{\ast}$) above allow the definition of a Latin trade as a partial Latin square that is not embedded in any particular Latin square. 

The partial Latin square $T'$ is said to be a {\em disjoint mate} of $T$, and the pair $(T,T')$ is called a {\em Latin bitrade} of $L$ of size $|T|$. 
It is said to be {\em primary} if, whenever $(U,U')$ is also a Latin bitrade in $L$ such that $U\subseteq T$ and $U'\subseteq T'$, then $U= T$ and $U'=T'$. 
Furthermore, the Latin trade $T$ is said to be {\em minimal} if, whenever $(U,U')$ is a Latin bitrade in $L$ such that $U\subseteq T$, then $U=T$. Finally, the Latin bitrade $(T,T')$ is said to be {\em orthogonal} if, whenever 
 $T[i,j]=T[i',j']$, with $i\neq i'$ and $j\neq j'$, then $T'[i,j]\neq T'[i',j']$.

Observe that if $L$ and $L'$ are distinct Latin squares of the same order, then $\mathrm{Ent}(L)\setminus \mathrm{Ent}(L')$ is a Latin trade with disjoint mate given by $\mathrm{Ent}(L')\setminus \mathrm{Ent}(L)$. Thus the study of Latin trades in Latin squares is equivalent to the study of the difference between Latin squares, sometimes known as the {\em Hamming distance} (\cite{CR, D1}). In particular, the minimum Hamming distance between group operation tables and other Latin squares has been a much studied topic 
\cite{BM1, CW, D2, DK, V}. Other applications of Latin trades include defining sets and randomization; see \cite{Cavenagh2008} for a survey.   

The Latin trade $T$ is said to be {\em $k$-homogeneous} if the set $\mathrm{Ent}(T)$ intersects each row, each column and each symbol in $L$ either zero or $k$ times. There exist  $3$-homogeneous Latin trades embedded in the Cayley table of the elementary abelian $2$-group $(\mathbb{Z}_2)^{2n}$, for all $n>1$  \cite{Cavenagh2004}. Particular constructions of $3$- and $4$-homogeneous Latin trades are described in \cite{Cavenagh2005,Cavenagh2005a}, while the existence of $k$-homogeneous Latin trades of certain sizes has been dealt with in \cite{Bean2005, Behrooz2011}.

In 2008, Cavenagh et al. \cite{CDH} described a method to construct Latin bitrades via certain finite groups. These Latin bitrades satisfy (1), (2) and ($3^{\ast}$) above and thus are not embedded in any particular Latin square. 
More specifically, rows, columns and entries are labelled as cosets of subgroups of the group under consideration. 
The Latin bitrades constructed in the following result are also {\em entry-transitive}. (We define entry-transitive in the next section.) 

\begin{theorem} {\rm (\cite[Theorem 2.14 and Lemma 3.2]{CDH})}  \label{drapl}
Let $G$ be a finite group with unit element $1$ and three elements $a,b,c\in G\setminus\{1\}$ such that  
\begin{itemize}
    \item[{\rm (G1)}]  $abc = 1$, and
    \item[{\rm (G2)}] $|<a> \cap <b>| = |<a> \cap <c>| = |<b> \cap <c>| = 1$.
\end{itemize}
Let $A=<a>$, $B=<b>$ and $C=<c>$. Then, the pair of partial Latin squares $(T^\circ,\,T^*)$, with respective set of entries 

\[\mathrm{Ent}(T^{\circ}):= \left\{\left(gA, gB, gC\right) \mid g \in G\right\}\hspace{0.5cm} \text{ and }\hspace{0.5cm}  \mathrm{Ent}(T^{\ast}):= \{(gA, gB, ga^{-1}C)\mid g \in G\},\]
is a Latin bitrade of size $|G|$, with $|G : A|$ rows (each with $|A|$ entries), $|G : B|$ columns (each with $|B|$ entries), and $|G : C|$ entries (each occurring $|C|$ times). It is primary, whenever $G=<a, b, c>$. It is orthogonal, whenever $|C\cap aCa^{-1}|=1$. 
\end{theorem}

\vspace{0.2cm}

In this paper, we generalize Theorem \ref{drapl} by describing a new method to construct Latin bitrades embedded in a given Latin square $L$, via subgroups of the autoparatopism group of $L$. This new method is described in Section \ref{Section:Automorphism}, where we also characterize the fundamental aspects of the Latin bitrades so constructed, as its size, homogeneity or orthogonality, amongst others. We finish that section with some examples illustrating all these results. A pair of more comprehensive examples based on the use of Mersenne primes, and on quadratic orthomorphisms over finite fields, are described in Section \ref{section:ort}.

\section{Preliminaries}
\label{Section:Preliminaries}

In this section, we describe some basic concepts, notations and results on partial Latin squares that are used throughout the paper. See \cite{Keedwell2015} for more details on this topic, and \cite{Cavenagh2008} for a survey on Latin bitrades.

\vspace{0.2cm}

Every Latin square of order $n$ constitutes the Cayley table of a {\em quasigroup} of the same order. That is, a pair $(Q,*)$ formed by a finite set $Q$ of $n$ elements that is endowed with a binary operation $*$, so that both equations $i*x=j$ and $y*i = j$ have unique solution $x,y\in Q$, for all $i,j\in Q$. Equivalently, the set $Q$ is endowed with right- and left-division. If the multiplication is associative, then the quasigroup is indeed a group. From here on, whenever there is no confusion, we denote by $Q$ the quasigroup $(Q,*)$, and we denote by $ij$ the product $i*j$, for all $i,j\in Q$. In addition, we denote by $L_*(Q)$ (by $L(Q)$, if there is no confusion) the Latin square describing the Cayley table of the quasigroup $(Q,*)$.

Two quasigroups $(Q,*)$ and $(Q',\circ)$ are {\em paratopic} if there are three bijections $f_1$, $f_2$ and $f_3$ from $Q$ to $Q'$, and a permutation $\pi$ in the symmetric group $S_3$, such that, 

\[f_{\pi(1)}(e_{\pi(1)})\circ f_{\pi(2)}(e_{\pi(2)})=f_{\pi(3)}(e_{\pi(e_3)}),\]
for all $(e_1,e_2,e_3)\in \mathrm{Ent}(L(Q))$. The tuple $\Theta=(f_1,f_2,f_3;\pi)$ is a {\em paratopism} from $(Q,*)$ to $(Q',\circ)$. It acts on the set $\mathrm{Ent}(L(Q))$ in the following way.
\begin{itemize}
    \item First, we permute the rows of $L(Q)$ according to the bijection $f_1$, its columns according to $f_2$, and its symbols according to $f_3$, giving rise to an intermediate Latin square of set of entries $\left\{\left(f_1(e_1),\,f_2(e_2),\,f_3(e_3)\right)\mid (e_1,e_2,e_3)\in\mathrm{Ent}(L(Q))\right\}$.

    \item Then, we permute the coordinates of each entry in this intermediate Latin square according to the permutation $\pi$, so that we get the set of entries
    
\begin{equation}\label{eq:paratopism}
\begin{array}{c}\mathrm{Ent}(L(Q'))=\\
\left\{\left(f_{\pi(1)}(e_{\pi(1)}), f_{\pi(2)}(e_{\pi(2)}), f_{\pi(3)}(e_{\pi(3)})\right)\mid (e_1,e_2,e_3)\in\mathrm{Ent}(L(Q))\right\}.
\end{array}
\end{equation}    
\end{itemize}

Thus, from now on, we denote

\[\Theta((e_1,e_2,e_3)):=(f_{\pi(1)}(e_{\pi(1)}), f_{\pi(2)}(e_{\pi(2)}), f_{\pi(3)}(e_{\pi(3)})).\]

Note that the composition of two paratopisms is again a paratopism. More specifically, if $\Theta'=(g_1,g_2,g_3;\rho)$ is a paratopism from the quasigroup $(Q',\circ)$ to a third quasigroup $(Q'',\diamond)$, then, for each entry $(e_1,e_2,e_3)\in \mathrm{Ent}(L(Q))$, we have that

\begin{align*}
\Theta'\Theta((e_1,e_2,e_3)) & =
\Theta'\left(f_{\pi(1)}(e_{\pi(1)}), f_{\pi(2)}(e_{\pi(2)}), f_{\pi(3)}(e_{\pi(3)})\right)\\& = \left(g_{\rho(1)}f_{\pi(\rho(1))}(e_{\pi(\rho(1))}), g_{\rho(2)}f_{\pi(\rho(2))}(e_{\pi(\rho(2))}), g_{\rho(3)}f_{\pi(\rho(3))}(e_{\pi(\rho(3))})\right)\\
& = \left(g_{\pi^{-1}(1)}f_1,g_{\pi^{-1}(2)}f_2,g_{\pi^{-1}(3)}f_3;\pi\rho\right)(e_1,e_2,e_3).
\end{align*}
That is,

\begin{equation}\label{equation_composition}
\Theta'\Theta:=\left(g_{\pi^{-1}(1)}f_1,g_{\pi^{-1}(2)}f_2,g_{\pi^{-1}(3)}f_3;\pi\rho\right).
\end{equation}

\vspace{0.5cm}

Let us illustrate the previous notions by focusing on the set $\mathcal{Q}_n$ of quasigroups of order $n$, with entries in the set of symbols $[n]:=\{0,1,\ldots,n-1\}$.

\vspace{0.25cm}

\begin{example} \label{example_autoparatopism}
Let $Q$, $Q'$ and $Q''$ be three quasigroups in $\mathcal{Q}_4$ of respective Cayley tables

\[L(Q)\equiv\begin{array}{|c|c|c|c|}\hline
1&3& 0& 2\\ \hline
0& 2& 1& 3\\ \hline
2& 0& 3& 1\\ \hline 
3& 1& 2& 0\\ \hline
\end{array}, \hspace{2cm} 
L(Q')\equiv\begin{array}{|c|c|c|c|}\hline
0 & 1& 2& 3\\ \hline
3& 2& 1& 0\\ \hline
2& 3& 0& 1\\ \hline 
1& 0& 3& 2\\ \hline
\end{array} \hspace{1cm} \text{and}\hspace{1cm}
L(Q'')\equiv\begin{array}{|c|c|c|c|}\hline
2&3& 0& 1\\ \hline
3& 2& 1& 0\\ \hline
1& 0& 3& 2\\ \hline 
0& 1& 2& 3\\ \hline
\end{array}.
\]
It can readily be checked that the tuples

\[\Theta = \left((01),(123),\mathrm{Id};(123)\right) \hspace{1cm} \text{and} \hspace{1cm} \Theta'= \left((0132),(12),(03);(23)\right)\]
are respective paratopisms from $Q$ to $Q'$, and from $Q'$ to $Q''$. Then, we have from (\ref{equation_composition}) that

\[\Theta'\Theta = \left((013),(01),(12);(12)\right)\]
is a paratopism from $Q$ to $Q''$. \hfill $\lhd$    
\end{example}

\newpage

Let $\Theta=(f_1,f_2,f_3;\pi)$ be a paratopism from a quasigroup $Q$ to a quasigroup $Q'$. If $f_1=f_2=f_3$ is the trivial permutation on $Q$, then $\Theta$ is said to be a {\em parastrophism} from $Q$ to its {\em parastrophe} $Q'$. 
In many papers the term ``conjugate'' is used for parastrophe. However, since group theory conjugates are used in this paper, we adhere to the term parastrophe to avoid confusion. 
If this is the case, the entries of $L(Q')$ coincide with those ones of $L(Q)$ after interchanging the role of rows, columns and symbols according to the permutation $\pi\in S_3$. The quasigroup $Q$ is {\em totally symmetric} if it is identical to each of its six parastrophes.

The tuple $\Theta=(f_1,f_2,f_3;\pi)$ is an {\em isotopism} if $\pi$ is the trivial permutation in the symmetric group $S_3$, in which case, the quasigroups $Q$ and $Q'$ are said to be {\em isotopic}, and the tuple $\Theta$ is denoted only by $(f_1,f_2,f_3)$. It is an {\em isomorphism} if $f_1=f_2=f_3:=f$, in which case we denote the triple $\Theta$ only by $f$. 

Furthermore, if $Q=Q'$, then the paratopism (respectively, the isotopism or the isomorphism) is an {\em autoparatopism} (respectively, an {\em autotopism} or an {\em automorphism}). We denote by $\mathrm{Apar}(Q)$ (respectively, $\mathrm{Atop}(Q)$ and $\mathrm{Aut}(Q)$) the set of autoparatopisms (respectively, the sets of autotopisms and automorphisms) of the quasigroup $Q$. The three sets have group structure. More specifically, the set of autoparatopisms has group structure with the composition described in (\ref{equation_composition}), while both the set of autotopisms and the set of automorphisms have group structure with the component-wise composition of permutations. They are respectively called the \emph{autoparatopism group}, the \emph{autotopism group} and the {\em automorphism group} of the quasigroup $Q$. In particular, $\mathrm{Aut}(Q)\leq\mathrm{Atop}(Q)\trianglelefteq\mathrm{Apar}(Q)$ (see \cite[Lemma 1]{Falcon2017}).

Let $H\in\{\mathrm{Apar}(Q),\,\mathrm{Atop}(Q),\,\mathrm{Aut}(Q)\}$. The {\em stabilizer} of an entry $e:=(e_1,e_2,e_3)\in \mathrm{Ent}(L(Q))$ with respect to a subgroup $G\leq H$ is the subgroup

\begin{align*}\mathrm{Stab}_G(e):& =\{(f_1,f_2,f_3;\pi)\in G\mid f_{\pi(i)}(e_{\pi(i)})=e_i, \text{ for all } i\in\{1,2,3\}\}\leq G\\
& = \mathrm{Stab}^{\mathrm{row}}_G(e)\cap \mathrm{Stab}^{\mathrm{col}}_G(e)\cap \mathrm{Stab}^{\mathrm{sym}}_G(e),
\end{align*}
where

\begin{align*}\mathrm{Stab}^{\mathrm{row}}_G(e):& =\{(f_1,f_2,f_3;\pi)\in G\mid f_{\pi(1)}(e_{\pi(1)})=e_1\}\leq G,\\
\mathrm{Stab}^{\mathrm{col}}_G(e): & =\{(f_1,f_2,f_3;\pi)\in G\mid f_{\pi(2)}(e_{\pi(2)})=e_2\}\leq G,\\
\mathrm{Stab}^{\mathrm{sym}}_G(e): & =\{(f_1,f_2,f_3;\pi)\in G\mid f_{\pi(3)}(e_{\pi(3)})=e_3\}\leq G.
\end{align*}
By abuse of notation, we also define the stabilizer of an element $i\in Q$ with respect to a subgroup $G\leq \mathrm{Aut}(Q)$ as the subgroup

\[\mathrm{Stab}_G(i):=\{f\in G\mid f(i)=i\}\leq G.\]

Based on (\ref{eq:paratopism}), all the previous concepts and notations are naturally translated to partial Latin squares. Furthermore, we say that two Latin bitrades $(T_1,T'_1)$ and $(T_2,T'_2)$ are paratopic if their respective components are paratopic. In this sense, we say that a partial Latin square $P$ is {\em entry-transitive} if the group $\mathrm{Apar}(Q)$ acts transitively on $\mathrm{Ent}(P)$. 

\section{The constructive method}
\label{Section:Automorphism}

Let $Q$ be a quasigroup. In this section, we generalize the construction of Latin bitrades described in Theorem \ref{drapl} via a subgroup of the autoparatopism group $\mathrm{Apar}(Q)$. The following lemma is useful to this end.

\vspace{0.25cm}

\begin{lemma}\label{lemma_rows} Let $G\leq \mathrm{Apar}(Q)$ and $e:=(e_1,e_2,e_3)\in\mathrm{Ent}(L(Q))$. Let $r\in Q$ be such that there is an autoparatopism $(f_1,f_2,f_3;\pi)\in G$ satisfying $f_{\pi(1)}(e_{\pi(1)})=r$. Then, there is a one-to-one correspondence between the set of autoparatopisms in $G$ mapping row $e_1$ to row $r$ in $L(Q)$, and the set $\mathrm{Stab}^{\mathrm{row}}_G(e)$.    
\end{lemma}

\begin{proof} Let $\Theta:=(f_1,f_2,f_3;\pi)$ and let $S$ denote the set of autoparatopisms in $G$ mapping row $e_1$ to row $r$ in $L(Q)$. Then, it is enough to consider the correspondence mapping each autoparatopism $\overline{\Theta}\in S$ to the autoparatopism $\overline{\Theta}^{-1}\Theta\in \mathrm{Stab}^{\mathrm{row}}_G(e)$. It is one-to-one, because each component of an autoparatopism is a 
 bijection.
\end{proof}

\vspace{0.2cm}

The following result describes the main constructive method to be studied in this paper.

\vspace{0.15cm}

\begin{theorem}\label{proposition_aut}
Let $G\leq \mathrm{Apar}(Q)$ and $\tau:=(e,\Theta,\overline{\Theta})\in \mathrm{Ent}(L(Q))\times G\times G$ be such that
\begin{enumerate}
    \item[{\rm (C1)}] $\Theta\in \mathrm{Stab}_G^{\mathrm{row}}(e)\setminus \mathrm{Stab}_G^{\mathrm{col}}(e)$;
    \item[{\rm (C2)}] $\overline{\Theta}\in\mathrm{Stab}_G^{\mathrm{col}}(e)$; and
    \item[{\rm (C3)}] $\overline{\Theta}^{-1}\Theta\in \mathrm{Stab}_G^{\mathrm{sym}}(e)$.
\end{enumerate}
If $e:=(e_1,e_2,e_3)$ and $\Theta:=(f_1,f_2,f_3;\pi)$, then the pair of partial Latin squares $(T_{\tau,G},T'_{\tau,G})$, with respective set of entries 

\[\mathrm{Ent}(T_{\tau,G}):=\left\{\Theta'(e)\mid \Theta'\in G\right\}\]
and

\[\mathrm{Ent}(T'_{\tau,G}):=\left\{\Theta'\left((e_1,e_2,f_{\pi(3)}(e_{\pi(3)}))\right)\mid \Theta'\in G\right\}\]
is a Latin bitrade of size $|G| / |\stab_G(e)|$.
\end{theorem}

\begin{proof} Since $G\leq \mathrm{Apar}(Q)$, we have that $T_{\tau,G}\subseteq L(Q)$. To determine the size of $T_{\tau,G}$, we suppose the existence of two autoparatopisms $\Theta_1:=(g_1,g_2,g_3;\rho)$ and $\Theta_2:=(g'_1,g'_2,g'_3;\rho')$ in $G$ such that $(g_{\rho(1)}(e_{\rho(1)}),\,g_{\rho(2)}(e_{\rho(2)}))=(g'_{\rho'(1)}(e_{\rho'(1)}),\,g'_{\rho'(2)}(e_{\rho'(2)}))$. From the condition of autoparatopism, we have that $\Theta_1(e),\, \Theta_2(e)\in\mathrm{Ent}(L(Q))$. It implies that $g_{\rho(3)}(e_{\rho(3)})=g'_{\rho'(3)}(e_{\rho'(3)})$, because any two distinct entries in a Latin square agree in at most one coordinate. Thus, $\Theta_1^{-1}\Theta_2\in \stab_G(e)$. Hence, $|T_{\tau,G}|=|G| / |\stab_G(e)|$. 

Next, we prove that $\mathrm{Ent}(T_{\tau,G})\cap \mathrm{Ent}(T'_{\tau,G})=\emptyset$. Otherwise, there are two autoparatopisms $\Theta_1,\,\Theta_2\in G$ such that $\Theta_1(e)=\Theta_2((e_1,e_2,f_{\pi(3)}(e_{\pi(3)})))\in \mathrm{Ent}(L(Q))$. Then, we have that $(e_1,e_2,f_{\pi(3)}(e_{\pi(3)}))\in \mathrm{Ent}(L(Q))$. Since $(e_1,e_2,e_3)\in\mathrm{Ent}(L(Q))$ and, again, any two distinct entries in a Latin square agree in at most one coordinate, it must be $f_{\pi(3)}(e_{\pi(3)}))=e_3$. But (C1) implies that $f_{\pi(3)}(e_{\pi(3)})=f_{\pi(1)}(e_{\pi(1)})f_{\pi(2)}(e_{\pi(2)})=e_1f_{\pi(2)}(e_{\pi(2)})\neq e_1e_2=e_3$, which is a contradiction. Thus, $\mathrm{Ent}(T_{\tau,G})\cap \mathrm{Ent}(T'_{\tau,G})=\emptyset$.

Now, we show that each non-empty row of $T_{\tau,G}$ and $T'_{\tau,G}$ contains the same set of symbols. From Lemma \ref{lemma_rows}, it suffices to show this for row $e_1$. To this end, since $\Theta\in \mathrm{Stab}^\mathrm{row}_G(e)$, note that 

\[\mathrm{Stab}^\mathrm{row}_G(e) =\left\{\Theta'\Theta\mid \Theta'\in \mathrm{Stab}^\mathrm{row}_G(e)\right\}.\]
Then, from (\ref{equation_composition}), row $e_1$ of
$T_{\tau,G}$ contains the set of symbols

\[\left\{g_{\rho(3)}\left(e_{\rho(3)}\right)\mid
\left(g_1,g_2,g_3;\rho\right)\in \mathrm{Stab}^\mathrm{row}_G(e)
\right\}=\]\[\left\{g_{\rho(3)}f_{\pi(\rho(3))}(e_{\pi(\rho(3))})\mid
\left(g_1,g_2,g_3;\rho\right)\in \mathrm{Stab}^\mathrm{row}_G(e)
\right\},\]
which is the set of symbols in row $e_1$ of $T'_{\tau,G}$.

Finally, we show that each non-empty column of $T_{\tau,G}$ and $T'_{\tau,G}$ contains the same set of symbols. From parastrophism and Lemma \ref{lemma_rows}, it suffices to show this for column $e_2$. To this end, suppose that $\overline{\Theta}:=(\overline{f}_1,\overline{f}_2,\overline{f}_3;\overline{\pi})$. From (C2), we have that 

\[\mathrm{Stab}^\mathrm{col}_G(e) =\left\{\Theta'\overline{\Theta}\mid \Theta'\in \mathrm{Stab}^\mathrm{col}_G(e)\right\}.\]
In addition, we have from (C3) that $\left(\Theta'\overline{\Theta}\right)^{-1}\left(\Theta'\Theta\right)\in \mathrm{Stab}_G^{\mathrm{sym}}(e)$, for all $\Theta'\in G$. From (\ref{equation_composition}), this is equivalent to the fact that $g_{\rho(3)}f_{\pi(\rho(3))}\left(e_{\pi(\rho(3))}\right)=g_{\rho(3)}\overline{f}_{\overline{\pi}(\rho(3))}\left(e_{\overline{\pi}(\rho(3))}\right)$, for all $\left(g_1,g_2,g_3;\rho\right)\in G$. Hence, column $e_2$ of $T_{\tau,G}$ contains the set of symbols

\[\begin{array}{c}\left\{g_{\rho(3)}\left(e_{\rho(3)}\right)\mid
\left(g_1,g_2,g_3;\rho\right)\in \mathrm{Stab}^\mathrm{col}_G(e)
\right\}\\
=\left\{
g_{\rho(3)}\overline{f}_{\overline{\pi}(\rho(3))}(e_{\overline{\pi}(\rho(3))})\mid
\left(g_1,g_2,g_3;\rho\right)\in \mathrm{Stab}^\mathrm{col}_G(e)
\right\}\\
=\left\{
g_{\rho(3)}f_{\pi(\rho(3))}(e_{\pi(\rho(3))})\mid
\left(g_1,g_2,g_3;\rho\right)\in \mathrm{Stab}^\mathrm{col}_G(e)
\right\}
\end{array}\]
which is the set of symbols in column $e_2$ of $T'_{\tau,G}$.
\end{proof}

\vspace{0.5cm}

\begin{example}\label{example_atop} Let $Q$ be the quasigroup of Cayley table

\[L(Q)\equiv\begin{array}{|c|c|c|c|} \hline
0 & 2 & 3 & 1\\ \hline 
1 & 3 & 2 & 0\\ \hline 
3 & 1 & 0 & 2\\ \hline 
2 & 0 & 1 & 3\\ \hline 
\end{array}.\]
Its autoparatopism group coincides with its autotopism group, which is generated by the autotopisms

\[\left((0123),\, (13),\,(0132)\right) \hspace{2cm}\text{and}\hspace{2cm}\left((23),\,(0132),\,(0213)\right).\]
As a consequence, $|\mathrm{Apar}(Q)|=|\mathrm{Atop}(Q)|=96$. Let us consider the entry $e:=(0,0,0)\in\mathrm{Ent}(L(Q))$ and the subgroup $G\leq \mathrm{Atop}(Q)$ of size $12$ that is generated by the following two autotopisms of $Q$.

\[\Theta=\left((123),\,(012),\,(023)\right) \hspace{4cm} \overline{\Theta}=\left((03)(12),\,\mathrm{Id},\,(02)(13)\right)\]
Conditions (C1)--(C3) in Theorem \ref{proposition_aut} hold and hence, the pair of partial Latin squares $(T_{\tau,G},T'_{\tau,G})$, with $\tau:=(e,\Theta,\overline{\Theta})$ constitute a Latin bitrade of $L(Q)$. Its entries are highlighted in bold type within $L(Q)$ as follows, with entries in $T'_{\tau,G}$ shown as subscripts. (It is the way in which we represent Latin bitrades throughout the paper.)

\[\begin{array}{|c|c|c|c|} \hline
\mathbf{0_2} & \mathbf{2_3} & \mathbf{3_0} & 1\\ \hline 
\mathbf{1_3} & \mathbf{3_2} & \mathbf{2_1} & 0\\ \hline 
\mathbf{3_1} & \mathbf{1_0} & \mathbf{0_3} & 2\\ \hline 
\mathbf{2_0} & \mathbf{0_1} & \mathbf{1_2} & 3\\ \hline 
\end{array}\]
\hfill $\lhd$ 
\end{example}

\vspace{0.5cm}

In what follows, we characterize some fundamental aspects of the Latin bitrade described in Theorem \ref{proposition_aut}. First, we establish the number of entries in the Latin trade $T_{\tau,G}$.

\begin{lemma}\label{lemma_atop}
The Latin trade $T_{\tau,G}$ satisfies the following statements.
\begin{enumerate}
\item Each non-empty row has 
$|\stab^{\mathrm{row}}_G(e)|/|\stab_G(e)|$ entries.
\item Each non-empty column has 
$|\stab^{\mathrm{col}}_G(e)|/|\stab_G(e)|$ entries
\item Each entry appears 
$|\stab^{\mathrm{sym}}_G(e)|/|\stab_G(e)|$ times
\end{enumerate}
Hence, the Latin trade  $T_{\tau,G}$ is $k$-homogeneous, if and only if $|\stab^{\mathrm{row}}_G(e)|=|\stab^{\mathrm{col}}_G(e)|=|\stab^{\mathrm{sym}}_G(e)|=k\cdot |\stab_G(e)|$.
\end{lemma}

\begin{proof} We prove the first statement. (The remaining ones hold by parastrophism, while the consequence holds readily from the three statements.) From Lemma \ref{lemma_rows}, it is enough to determine the number of entries in the row $e_1$. As we have already observed in the proof of Theorem \ref{proposition_aut}, row $e_1$ of $T_{\tau,G}$ contains the set of symbols

\[\left\{g_{\rho(3)}(e_{\rho(3)})\mid
\left(g_1,g_2,g_3;\rho\right)\in \mathrm{Stab}^\mathrm{row}_G(e)
\right\}.\]

Thus, if there are two distinct autoparatopisms $\Theta_1:=\left(g_1,g_2,g_3;\rho\right)$ and $\Theta_2:=\left(g'_1,g'_2,g'_3;\rho'\right)$ in $\mathrm{Stab}^\mathrm{row}_G(e)$ such that $g_{\rho(3)}(e_{\rho(3)})=g'_{\rho'(3)}(e_{\rho'(3)})$, then $\Theta_2^{-1}\Theta_1\in
\mathrm{Stab}^{\mathrm{sym}}_G(e)$. Since $\Theta_2^{-1}\Theta_1\in
\mathrm{Stab}^{\mathrm{row}}_G(e)$, we have that $\Theta_2^{-1}\Theta_1\in
\mathrm{Stab}_G(e)$, because any two distinct entries in a Latin trade agree in at most one coordinate. As a consequence, row $e_1$ contains $|\stab^{\mathrm{row}}_G(e)|/|\stab_G(e)|$ entries.    
\end{proof}

\vspace{0.5cm}

Now, we characterize the orthogonality of our Latin bitrade.

\vspace{0.25cm}

\begin{lemma}\label{lemma_Autot_Orthog}
The Latin bitrade $(T_{\tau,G}, 
T'_{\tau,G})$ is orthogonal if and only if
\begin{equation}\label{eq:orthog}\mathrm{Stab}^{\mathrm{sym}}_G(e)\cap \mathrm{Stab}^{\mathrm{sym}}_G((e_1,e_2,f_{\pi(3)}(e_{\pi(3)})))
\subseteq \mathrm{Stab}_G(e).
\end{equation}
\end{lemma} 
 
\begin{proof}
First suppose that (\ref{eq:orthog}) is true. 
Let $\Theta_1:=\left(g_1,g_2,g_3;\rho\right)$ and $\Theta_2:=\left(g'_1,g'_2,g'_3;\rho'\right)$ be two autoparatopisms in $G$ such that 

\[g_{\rho(1)}\left(e_{\rho(1)}\right)\neq g'_{\rho'(1)}\left(e_{\rho'(1)}\right),\]
\[g_{\rho(2)}\left(e_{\rho(2)}\right)\neq g'_{\rho'(2)}\left(e_{\rho'(2)}\right)\]
and 

\[g_{\rho(3)}\left(e_{\rho(3)}\right)= g'_{\rho'(3)}\left(e_{\rho'(3)}\right).\]
Then, 

  \[\Theta_2^{-1}\Theta_1\in \mathrm{Stab}^{\mathrm{sym}}_G(e)\setminus \mathrm{Stab}_G(e).\]
From (\ref{eq:orthog}), we have that 

\[\Theta_2^{-1}\Theta_1\not\in
\mathrm{Stab}^{\mathrm{sym}}_G((e_1,e_2,f_{\pi(3)}(e_{\pi(3)}))).\]
Hence, $g_{\rho(3)}f_{\pi(\rho(3))}(e_{\pi(\rho(3))})\neq g'_{\rho(3)}f_{\pi(\rho(3))}(e_{\pi(\rho(3))})$. Thus, 
the Latin bitrade $(T_{\tau,G}, T'_{\tau,G})$ is orthogonal. 

Conversely, suppose the existence of an autoparatopism 

\[(g_1,g_2,g_3;\rho)\in \left(\mathrm{Stab}^{\mathrm{sym}}_G(e)\cap \mathrm{Stab}^{\mathrm{sym}}_G((e_1,e_2,f_{\pi(3)}(e_{\pi(3)})))\right)\setminus \mathrm{Stab}_G(e).\]
Since $(g_1,g_2,g_3;\rho)\in \mathrm{Stab}^{\mathrm{sym}}_G(e)$, we have that $g_{\rho(3)}(e_{\rho(3)})=e_3$. As a consequence, since $(g_1,g_2,g_3;\rho)\not\in \mathrm{Stab}_G(e)$ and every two distinct entries in a Latin trade agree in at most one coordinate, it must be that $g_{\rho(1)}(e_{\rho(1)})\neq e_1$ and $g_{\rho(2)}(e_{\rho(2)})\neq e_2$. Further,  since 
$(g_1,g_2,g_3;\rho)\in \mathrm{Stab}^{\mathrm{sym}}_G((e_1,e_2,f_{\pi(3)}(e_{\pi(3)})))$, we have that

\[g_{\rho(3)}f_{\pi(\rho(3))}(e_{\pi(\rho(3))})=f_{\pi(3)}(e_{\pi(3)}).\]

In summary, 
$(e_1,e_2,e_3)$ and $(g_{\rho(1)}(e_{\rho(1)}), 
g_{\rho(2)}(e_{\rho(2)}),e_3)$ are distinct entries of 
$T_{\tau,G}$ while 
$(e_1,e_2,f_{\pi(3)}(e_{\pi(3)}))$ and $(g_{\rho(1)}(e_{\rho(1)}), 
g_{\rho(2)}(e_{\rho(2)}),f_{\pi(3)}(e_{\pi(3)}))$ are distinct entries of 
$T'_{\tau,G}$ containing the same symbol. 
This 
contradicts orthogonality. 
\end{proof}

\vspace{0.5cm}

Next, we characterize those intermediate subgroups between $G$ and $\mathrm{Apar}(Q)$ under whose action the Latin trade $T_{\tau,G}$ is a block.

\vspace{0.15cm}

\begin{lemma}
\label{lemma_block}
Under the assumptions of Theorem \ref{proposition_aut}, let $B\leq \mathrm{Apar}(Q)$ be such that $G\leq B$. If we denote $S_B:=\mathrm{Stab}_B(e)$, then the Latin trade $T_{\tau,G}$ is a block under the action of $B$ if and only if $S_BG= GS_B$. 
\label{blockproof}
\end{lemma}

\begin{proof}
Suppose firstly that 
$S_BG= GS_B$. 
Let $\Theta'\in B$ be such that 
$\Theta'(e)\in T_{\tau,G}$. 
Then, there exists an autoparatopism $\Theta''\in G$ such that $\Theta'(e) = \Theta''(e)$. Thus, ${\Theta''}^{-1}\Theta'\in S_B$. Letting ${\Theta''}^{-1}\Theta'=\Theta_1$, we have 
$\Theta'=\Theta''\Theta_1$. 

 Next, let $\Theta_2\in G$. 
Since $S_BG=GS_B$, $\Theta_2^{-1}\Theta_1^{-1}\Theta_2\in S_B$. Thus, 
\begin{eqnarray*}
\Theta_2^{-1}\Theta_1^{-1}\Theta_2(e) & = &   e \\
\Leftrightarrow \  \Theta'\Theta_2(\Theta_2^{-1}\Theta_1^{-1}\Theta_2(e)) & = &  \Theta'\Theta_2(e) \\
\Leftrightarrow \  
 \Theta''\Theta_1\Theta_2(
\Theta_2^{-1}\Theta_1^{-1}\Theta_2(e)) & = &  \Theta'\Theta_2(e)  \\
\Leftrightarrow \  
\Theta''\Theta_2(e) & = &  \Theta'\Theta_2(e).   
\end{eqnarray*}
Hence, since $\Theta''\Theta_2\in G$, we have that $\Theta'\Theta_2(e)\in T_{\tau,G}$.

In summary, whenever $\Theta'\in B$ satisfies 
$\Theta'(e)\in T_{\tau,G}$, we have that $\Theta'\Theta_2(e)\in T_{\tau,G}$,
for each $\Theta_2\in G$. 
Thus, the Latin trade $T_{\tau,G}$ is a block under the action of $B$. 

Conversely, suppose that 
the Latin trade $T_{\tau,G}$ is a block under the action of $B$.
Let $\Theta_1\in S_B$ and $\Theta_2\in  G$. 
By the definition of
a block, since $\Theta_1(e)=e$, we have that $\Theta_1(\Theta_2(e))\in T_{\tau,G}$ and thus 
$\Theta_1(\Theta_2(e))
=\Theta_3(e)$ for some $\Theta_3\in G$. 
In turn,
$\Theta_3^{-1}\Theta_1\Theta_2\in S_B$.
Thus $\Theta_1\Theta_2=\Theta_3\Theta_4$ for some $\Theta_4\in S_B$. 
It follows that 
$S_BG\subseteq GS_B$. 
By elementary group theory,
$|S_BG|=|GS_B|$, so we have 
that 
$S_BG= GS_B$. 
\end{proof}

\vspace{0.5cm}

We have already indicated that our proposal generalizes the constructive method described in Theorem \ref{drapl}. More precisely, if the subgroup $G\leq \mathrm{Apar}(Q)$ constitutes indeed a subgroup in the automorphism group $\mathrm{Aut}(Q)$, then Theorem \ref{drapl} and Theorem \ref{proposition_aut} can be identified under certain conditions, which we show in the next result.

\vspace{0.15cm}

\begin{theorem}\label{theorem_generalization}
Let $G\leq \mathrm{Aut}(Q)$ and $\tau:=(e,\alpha,\overline{\alpha})\in \mathrm{Ent}(L(Q))\times G\times G$ be such that 
\begin{enumerate}
    \item[{\rm (C$1'$)}] $\alpha\in \mathrm{Stab}^{\mathrm{row}}_G(e)\setminus \mathrm{Stab}^{\mathrm{col}}_G(e)$;
    \item[{\rm (C$2'$)}] $\overline{\alpha}\in\mathrm{Stab}^{\mathrm{col}}_G(e)$; and
    \item[{\rm (C$3'$)}] $\overline{\alpha}^{-1}\alpha\in \mathrm{Stab}^{\mathrm{sym}}_G(e)$.
\end{enumerate}

Suppose in turn that
\begin{enumerate}
\item[{\rm (B1)}]   $\mathrm{Stab}^{\mathrm{row}}_G(e)=<\alpha>$;
\item[{\rm (B2)}]   $\mathrm{Stab}^{\mathrm{col}}_G(e)=<\overline{\alpha}>$;
\item[{\rm (B3)}]   $\mathrm{Stab}^{\mathrm{sym}}_G(e)=<\overline{\alpha}^{-1}\alpha>$;
\item[{\rm (B4)}]   $|<\alpha>\cap <\overline{\alpha}>|=1$.
\end{enumerate}

Let $a:=\alpha^{-1}$, $b:=\overline{\alpha}$, $c:=\overline{\alpha}^{-1}\alpha$,
$A:=<a>$, $B:=<b>$ and $C:=<c>$.
Then, the Latin bitrade $(T_{\tau,G},\, T'_{\tau,G})$ described as in Theorem \ref{proposition_aut} is isotopic to the Latin bitrade $(T^\circ,\, T^*)$ described as in Theorem \ref{drapl} via the pair of isotopisms $(\Theta,\, \Theta')$ that is defined as follows.

\[\Theta g(e):=(gA,gB,gC)\]
and

\[\Theta' g(e):=(gA,gB,ga^{-1}C),\]
for all automorphism $g\in G$.
\end{theorem}

\begin{proof}
First, suppose that
$\Theta g_1(e)=\Theta g_2(e)$, for some $g_1,g_2\in G$.
Then $g_2^{-1}g_1\in \mathrm{Stab}_G(e)$. From (B4),
$g_1=g_2$. Thus, $\Theta$ is an injection.
Since $|T_{\tau,G}|=|T^{\circ}|=|G|$ by Theorem \ref{drapl} and Theorem \ref{proposition_aut}, $\Theta$ is in turn a bijection.

We next show that $\Theta$ preserves rows, columns and entries. To this end, suppose that both entries
$g_1(e)$
and
$g_2(e)$
share the same row in $T_{\tau,G}$, for some $g_1, g_2\in G$.
Then, $g_2^{-1}g_1\in \mathrm{Stab}^{\mathrm{row}}_G(e)=A$, and
thus, $g_1A=g_2A$. Hence,
$\Theta g_1(e)$ and
$\Theta g_2(e)$
lie in the same row of $T^{\circ}$.
Similarly,
if two entries in $T_{\tau,G}$ share the same column (or entry), then
this property is preserved under the map $\Theta$.
Finally, since $\alpha=a^{-1}$, it follows similarly that $\Theta'(T'_{\tau,G})=T^{\ast}$.
\end{proof}

\vspace{0.5cm}

Let us finish this section with some illustrative examples.

\vspace{0.15cm}

\begin{example}\label{example_atop1} Let $Q$ be the additive group $((\mathbb{Z}_2)^2,+)$, whose Cayley table is the Latin square

\[\begin{array}{|c|c|c|c|} \hline
0 & 1 & 2 & 3\\ \hline
1 & 0 & 3 & 2\\ \hline
2 & 3 & 0 & 1\\ \hline
3 & 2 & 1 & 0\\ \hline
\end{array}\]
The autotopism group $\mathrm{Atop}(Q)$ has size $96$. It is generated, for example, by the automorphism $(123)$ and the autotopism $((01),\,(0213),\,(0312))$.

Now, we consider the subgroup $G\leq \mathrm{Atop}(Q)$ of size $12$ that is generated by the autotopisms

\[\Theta:=((123),(013),(013)) \hspace{1cm}\text{and}\hspace{1cm} \overline{\Theta}:= ((012),(132),(012)).\]
Conditions (C1)--(C3) in Theorem \ref{proposition_aut} hold and hence, we can construct the Latin bitrade $(T_{\tau,G},\,T'_{\tau,G})$, with $e:=(0,0,0)\in\mathrm{Ent}(L(Q))$ and $\tau:=(e,\Theta,\overline{\Theta})$. Since $\mathrm{Stab}_G(e)=\{\mathrm{Id}\}$, we have that $|T_{\tau,G}|=|G|=12$. Even more, since $|\mathrm{Stab}^{\mathrm{row}}_G(e)|=|\mathrm{Stab}^{\mathrm{col}}_G(e)|=|\mathrm{Stab}^{\mathrm{sym}}_G(e)|=3$, Lemma \ref{lemma_atop} ensures that both Latin trades $T_{\tau,G}$ and $T'_{\tau,G}$ are $3$-homogeneous. Their respective $12$ entries are highlighted in bold type within $L(Q)$ as follows.

\[\begin{array}{|c|c|c|c|} \hline
{\bf 0_1} & {\bf 1_3} & 2 & {\bf 3_0}\\ \hline
{\bf 1_2} & 0 & {\bf 3_1} & {\bf 2_3}\\ \hline
{\bf 2_0} & {\bf 3_2} & {\bf 0_3} & 1\\ \hline
3 & {\bf 2_1} & {\bf 1_0} & {\bf 0_2}\\ \hline
\end{array}\]

It is orthogonal from Lemma \ref{lemma_Autot_Orthog}, because $\mathrm{Stab}^{\mathrm{sym}}_G(e)\cap \mathrm{Stab}^{\mathrm{sym}}_G((0,0,1))=\mathrm{Stab}_G(e)=\left\{\mathrm{Id}\right\}$. 
\hfill $\lhd$
\end{example}

\vspace{0.5cm}

\begin{example}\label{example_atop2}  Let $Q$ be the totally symmetric group $(({\mathbb Z}_2)^3,+)$, whose  Cayley table is the Latin square

$$L(Q)\equiv\begin{array}{|c|c|c|c|c|c|c|c|}
\hline
0 & 1 & 2 & 3 & 4 & 5 & 6 & 7 \\
\hline
1 & 0 & 3 & 2 & 5 & 4 & 7 & 6 \\
\hline
2 & 3 & 0 & 1 & 6 & 7 & 4 & 5 \\
\hline
3 & 2 & 1 & 0 & 7 & 6 & 5 & 4 \\
\hline
 4 & 5 & 6 & 7 & 0 & 1 & 2 & 3 \\
\hline
 5 & 4 & 7 & 6 & 1 & 0 & 3 & 2\\
\hline
6 & 7 & 4 & 5 & 2 & 3 & 0 & 1\\
\hline
7 & 6 & 5 & 4 & 3 & 2 & 1 & 0 \\
\hline
\end{array}.$$
The automorphism group $\mathrm{Aut}(Q)$ has size $168$. It is generated, for example, by the automorphisms $(1735)(46)$ and $(1736452)$. Furthermore, the autotopism group $\mathrm{Atop}(Q)$ has size $10752$. It is generated, for example, by the autotopisms

\[\left((0265)(1374),\, (0573)(1462),\, (0716)(23) \right)\hspace{0.5cm}\text{and}\hspace{0.5cm} \left((0526713),\, (0361542),\,(0647251) \right).\]
Now, let us consider the subgroup $G\leq \mathrm{Atop}(Q)$ of size $32$ that is generated by the autotopisms

\[\Theta:=\left((1643)(27),
(0621)(3745),(0621)(3745)\right)\]
and

\[\overline{\Theta}:=\left((0674)(1532),(1346)(27),(0674)(1532)\right).\]

Conditions (C1)--(C3) in Theorem \ref{proposition_aut} hold and hence, we can construct the Latin bitrade $(T_{\tau,G},\,T'_{\tau,G})$, with  $e:=(0,0,0)$ and $\tau:=(e,\Theta,\overline{\Theta})$. 
Since $\mathrm{Stab}_G(e)=\{\mathrm{Id}\}$, we have that $|T_{\tau,G}|=|G|=32$. 

Even more, since $|\mathrm{Stab}^{\mathrm{row}}_G(e)|=|\mathrm{Stab}^{\mathrm{col}}_G(e)|=|\mathrm{Stab}^{\mathrm{sym}}_G(e)|=4$, Lemma \ref{lemma_atop} ensures that both Latin trades $T_{\tau,G}$ and $T'_{\tau,G}$ are $4$-homogeneous. Their respective $32$ entries are highlighted in bold type within $L(Q)$ as follows.

$$\begin{array}{|c|c|c|c|c|c|c|c|}
\hline
{\bf 0_6} & {\bf 1_0} & {\bf 2_1} & 3 & 4 & 5 & {\bf 6_2} & 7 \\
\hline
1 & {\bf 0_3} & {\bf 3_2} & {\bf 2_4} & 5 & {\bf 4_0} & 7 & 6 \\
\hline
2 & {\bf 3_7} & 0 & 1 & 6 & {\bf 7_4} & {\bf 4_5} & {\bf 5_3} \\
\hline
3 & 2 & {\bf 1_5} & 0 & {\bf 7_1} & {\bf 6_7} & {\bf 5_6} & 4 \\
\hline
{\bf 4_0} & 5 & 6 & 7 & {\bf 0_3} & 1 & {\bf 2_4} & {\bf 3_2} \\
\hline
 5 & 4 & 7 & {\bf 6_2} & {\bf 1_0} & {\bf 0_6} & 3 & {\bf 2_1} \\
\hline
{\bf 6_7} & {\bf 7_1} & 4 & {\bf 5_6} & 2 & 3 & 0 & {\bf 1_5} \\
\hline
{\bf 7_4} & 6 & {\bf 5_3} & {\bf 4_5} & {\bf 3_7} & 2 & 1 & 0 \\
\hline
\end{array}$$

\vspace{0.25cm}

Lemma \ref{lemma_Autot_Orthog} implies that this Latin bitrade is not orthogonal, because $|\mathrm{Stab}_G(e)|=1$ and $\mathrm{Stab}^{\mathrm{sym}}_G(e)\cap \mathrm{Stab}^{\mathrm{sym}}_G((0,0,6))=\left\{\left((05)(14)(27)(36),\,(05)(14)(27)(36),\,\mathrm{Id}\right),\, \mathrm{Id}\right\}$. \hfill $\lhd$
\end{example}

\vspace{0.5cm}

The $4$-homogeneous Latin trade described in Example \ref{example_atop2} was first observed in \cite{Cavenagh2005a}, as a singular example not part of a general construction. Therein, it was noted that this Latin trade gives the minimum Hamming distance, 32, from $(({\mathbb Z}_2)^3,+)$ to a Latin square containing no $2\times 2$ subsquares. It is also a minimal Latin trade. Adding any new entry to the Latin trade yields a minimal defining set of the Latin square $(({\mathbb Z}_2)^3,+)$.

\vspace{0.5cm}

\begin{example}\label{example_apar} Since the group $Q$ described in Example \ref{example_atop2} is totally symmetric, we have that $|\mathrm{Apar}(Q)|=6\cdot |\mathrm{Atop}(Q)|=64512$. Moreover, the autoparatopism group $\mathrm{Apar}(Q)$ is generated by the autoparatopisms

\[\left((0265)(1374),\, (0573)(1462),\, (0716)(23);\,(123)\right)\]
and 

\[\left((0526713),\, (0361542),\,(0647251);(12) \right).\]

Let us consider again the entry $e=(0,0,0)\in \mathrm{Ent}(L(Q))$, as in Example \ref{example_atop2}, and let $G'\leq \mathrm{Apar}(Q)$ be the subgroup that is generated by the following three autoparatopisms. 

\begin{align*}
\Theta_1:& =\left( 
(0642)(1753), (26)(37),
(0642)(1753);(12)\right) \\
\Theta_2: & =\left((0213)(4657),(0213)(4657),(23)(67);\mathrm{Id}\right)\\
\Theta_3: & =\left((04)(15)(26)(37),(05)(14)(36)(27),(01)(23)(45)(67);\mathrm{Id}\right)
\end{align*}

Let $\tau':=(e,\Theta_1^2\Theta_2\Theta_1,\Theta_1)$. Since $\left|\mathrm{Stab}_{G'}(e)\right|=6$ and $|\mathrm{Stab}^{\mathrm{row}}_{G'}(e)|=|\mathrm{Stab}^{\mathrm{col}}_{G'}(e)|=|\mathrm{Stab}^{\mathrm{sym}}_{G'}(e)|=24$, then Lemma \ref{lemma_atop} implies that both Latin trades $T_{\tau',G'}$ and $T'_{\tau',G'}$ are $4$-homogeneous. Their respective $32$ entries are highlighted in bold type within $L(Q)$ as follows.

$$\begin{array}{|c|c|c|c|c|c|c|c|}
\hline
{\bf 0_6} & 1 & {\bf 2_0} & 3 & 4 & {\bf 5_2} & {\bf 6_5} & 7 \\
\hline
1 & {\bf 0_6} & 3 & {\bf 2_0} & {\bf 5_2} & 4 & 7 & {\bf 6_5} \\
\hline
2 & {\bf 3_0} & {\bf 0_7} & 1 & 6 & {\bf 7_5} & 4 & {\bf 5_3} \\
\hline
{\bf 3_0} & 2 & 1 & {\bf 0_7} & {\bf 7_5} & 6 & {\bf 5_3} & 4 \\
\hline
 {\bf 4_3} & 5 & 6 & {\bf 7_4} & 0 & {\bf 1_7} & 2 & {\bf 3_1} \\
\hline
 5 & {\bf 4_3} & {\bf 7_4} & 6 & {\bf 1_7} & 0 & {\bf 3_1} & 2 \\
\hline
{\bf 6_4} & 7 & {\bf 4_2} & 5 & {\bf 2_1} & 3 & 0 & {\bf 1_6} \\
\hline
7 & {\bf 6_4} & 5 & {\bf 4_2} & 3 & {\bf 2_1} & {\bf 1_6} & 0 \\
\hline
\end{array}$$
\hfill $\lhd$
\end{example}

\section{Further examples}

We finish our study by illustrating our constructive proposal with a pair of more comprehensive examples. They refer to the construction of Latin bitrades arisen from Mersenne primes and from quadratic orthomorphisms.

\subsection{Construction of Latin trades via Mersenne primes}

Let $p$ be a Mersenne prime; that is, $p=2^q-1$, where both $p$ and $q$ are primes. Assume also that $q\geq 3$.   
If ${\mathbb F}$ denotes the finite field of order $2^q$, then let $a$ be a primitive element in 
${\mathbb F}$; let $i$ be the unique solution
modulo $p$ 
to 
$a^{2i}+a^2+a+1=0$; and let $Q$ be the additive group in the same field. Suppose furthermore the existence of a positive integer $j$ such that $2^j+i-1$ is divisible by $p$.
 
Let $\omega,\,\alpha\in\mathrm{Aut}(Q)$ be defined by $\omega(x)=ax$ and
$\alpha(x)=x^{(p+1)/2}$. (Note here that $\alpha^{-1}(x)=x^2$ and hence, $\alpha^{-1}$ (and thus $\alpha$) is an automorphism of $Q$, because $(x+y)^2=x^2+y^2$ over ${\mathbb F}$.) In addition, observe that $\omega$ has order $p-1$;  $\alpha^{-1}$ (and hence 
$\alpha$) has order $q$, and 
$\alpha \omega^2  =  \omega\alpha$, 
composing from right to left. From \cite{CDH}, for example, we have that $\omega$ and $\alpha$ generate a group $G\leq \mathrm{Aut}\left(((\mathbb{Z}_2)^q,+)\right)$ of 
size $pq$. 

Next, let $e:=(1,a,a+1)\in\mathrm{Ent}(L)$. Then, $\mathrm{Stab}_G(e)=\{\mathrm{Id}\}$ and $|\mathrm{Stab}_G^{\mathrm{row}}(e)|=|\mathrm{Stab}_G^{\mathrm{col}}(e)|=|\mathrm{Stab}_G^{\mathrm{sym}}(e)|=q$. In addition, $\alpha\in\mathrm{Stab}_G^{\mathrm{row}}(e)\setminus\mathrm{Stab}_G^{\mathrm{col}}(e)$. Further, let $\overline{\alpha}=\omega^i\alpha^{-j}$. Since $2^j+i-1$ is divisible by $p$, we have that  

$$
\overline{\alpha}(a) 
 =    
\omega^i\alpha^{-j}(a) 
 =  
\omega^i(a^{2^j})
= a^{2^j+i}=a.$$
So, $\overline{\alpha}\in \mathrm{Stab}_G^{\mathrm{col}}(e)$. Moreover, 

\begin{align*}\alpha^{-1}\overline{\alpha}(a+1) &  
= \alpha^{-1}(\omega^i\alpha^{-j}(a+1)) = \alpha^{-1}(
\omega^i(
a^{2^{j}}+1)) 
= \alpha^{-1}(a^{2^{j}+i}+a^i) =\\
& = \alpha^{-1}(a+a^i)=a^2+a^{2i}.
\end{align*}

From the above assumption on $a$, we have that $\overline{\alpha}^{-1}\alpha\in\mathrm{Stab}_G^{\mathrm{sym}}(e)$. As a consequence, the next result follows readily from Theorem \ref{proposition_aut} and Lemma \ref{lemma_atop}, for the subgroup $G$ and the tuple $(e,\alpha,\overline{\alpha})$.

\vspace{0.25cm}

\begin{theorem}\label{theorem_Mersenne}
Let $p=2^q-1$ be 
a Mersenne prime
such that 
$a^{2i}+a^2+a+1=0$, for some positive integer $i$ and  
some primitive element
 $a$, over the field of order $2^q$. Suppose furthermore that there exists a positive integer $j$ such that 
 $2^j+i-1$ is divisible by $p$. 
Then, there exists a $q$-homogeneous minimal 
Latin trade in $(({\mathbb Z}_2)^q,+)$ of size 
$pq$. 
\end{theorem}

\vspace{0.2cm}

\begin{example}\label{example_Mersenne} The hypothesis of Theorem \ref{theorem_Mersenne} holds for $q=3$, $i=6$, $j=1$ and $a^3=a+1$. Let $Q$ be the additive group in the finite field of order $8$. It is isomorphic to the totally symmetric group $(({\mathbb Z}_2)^3,+)$, whose  Cayley table is the Latin square $L(Q)$ described in Example \ref{example_atop2}. Particularly, the elements $a$, $a^2$, $a^3=a+1$, $a^4=a^2+a$, $a^5=a^2+a+1$ and $a^6=a^2+1$ are respectively represented by $2$, $6$, $3$, $4$, $5$ and $7$. 

From Theorem \ref{theorem_Mersenne}, we can construct a $3$-homogeneous minimal Latin trade in $Q$ of size $21$. To this end, we consider the automorphisms $\omega(x)=ax$ and $\alpha(x)=x^4$ in $\mathrm{Aut}(Q)$, which are respectively represented by the permutations $(1263457)$ and $(246)(357)$. In addition, we define the automorphism $\overline{\alpha}(x)=\omega^6\alpha^{-1}(x)$ in $\mathrm{Aut}(Q)$, which is represented by the permutation
$(174)(356)$. Finally, we consider the entry $e:=(1,2,3)\in\mathrm{Ent}(L(Q))$. Our construction yields the following Latin bitrade $\left(T_{\tau,G},\,T'_{\tau,G}\right)$ in $L(Q)$, where $G:=\langle\,\omega,\,\alpha\,\rangle$ and $\tau:=(e,\alpha,\overline{\alpha})$. 

$$\begin{array}{|c|c|c|c|c|c|c|c|}
\hline
0 & 1 & 2 & 3 & 4 & 5 & 6 & 7 \\
\hline
1 & 0 & {\bf 3_5} & 2 & {\bf 5_7} & 4 & {\bf 7_3} & 6 \\
\hline
2 & 3 & 0 & {\bf 1_4} & 6 & {\bf 7_1} & {\bf 4_7} & 5 \\
\hline
3 & {\bf 2_6} & 1 & 0 & {\bf 7_2} & {\bf 6_7} & 5 & 4 \\
\hline
 4 & 5 & {\bf 6_3} & 7 & 0 & {\bf 1_6} & 2 & {\bf 3_1} \\
\hline
 5 & {\bf 4_2} & 7 & 6 & 1 & 0 & {\bf 3_4} & {\bf 2_3} \\
\hline
6 & 7 & 4 & {\bf 5_1} & {\bf 2_5} & 3 & 0 & {\bf 1_2} \\
\hline
7 & {\bf 6_4} & {\bf 5_6} & {\bf 4_5} & 3 & 2 & 1 & 0 \\
\hline
\end{array}$$

\vspace{0.25cm}

Lemma \ref{lemma_Autot_Orthog} implies that this Latin bitrade is orthogonal, because 

\[\mathrm{Stab}^{\mathrm{sym}}_G(e)\cap \mathrm{Stab}^{\mathrm{sym}}_G((1,2,5))=\{\mathrm{Id}\}=\mathrm{Stab}_G(e).\]
Furthermore, the subgroup $\mathrm{Stab}_{\aut(Q)}(e)\leq \aut(Q)$ has size four. It is generated by the automorphisms $(46)(57)$ and $(45)(67)$. Then, 

\[\left(|\mathrm{Aut}(Q)|\,/\,|\mathrm{Stab}_{\aut(Q)}(e)|\right)\,/\,|T_{\tau,G}|=2.\]
Indeed, under the action of the subgroup $G$, there are two disjoint and isomorphic copies of the Latin trade $T_{\tau,G}$. These turn out to be the original $T_{\tau,G}$ and its transpose, which arises, for instance, by replacing the  entry $e$ in the tuple $\tau$ by the entry $(1,3,2)$. (Note here that $\mathrm{Stab}_{\aut(Q)}((1,3,2))=\mathrm{Stab}_{\aut(Q)}(e)$.) Alternatively, if we add the generator $(12)(56)\in\mathrm{Aut}(Q)$ to the group $G$ to create a group $G'$ of order $42$, then the latter induces two orbits on the set 
$\mathrm{Ent}(L(T_{\tau,G}))=\{(i,j,i+j)\in \mathrm{Ent}(L(Q)) \mid 0\not\in \{i,j,i+j\}\}\subset\mathrm{Ent}(L(Q))$. One is the above Latin trade $T_{\tau,G}$. The other one is its transpose, which is obviously also a Latin trade.

Note that the subgroup $\mathrm{Stab}_{\aut(Q)}(e)\leq \mathrm{Aut}(Q)$ does not stabilize the Latin trade $T_{\tau,G}$ as a set, because $\omega$ maps, for instance, the entry $(2,6,4)\in\mathrm{Ent}(T_{\tau,G})$ to the entry $(0,4,6)\not\in\mathrm{Ent}(T_{\tau,G})$. Since the group $G$ is the stabilizer in $\mathrm{Aut}(Q)$ of the Latin trade $T_{\tau,G}$, the  latter is not a block in the action of $\aut(Q)$ on $L(Q)$. For instance, the automorphism $(1746325)\in\mathrm{Aut}(Q)$ maps the entry $e$ to the entry $(7,5,2)\not\in\mathrm{Ent}(T_{\tau,G})$.

\hfill $\lhd$
\end{example}

\subsection{Construction of Latin trades via quadratic orthomorphisms}\label{section:ort}

Theorem \ref{proposition_aut} can be implemented to construct Latin bitrades from quadratic orthomorphisms over finite fields. Recall here that an {\em orthomorphism} over a finite field ${\mathbb F}_q$, with $q$ a prime power, is a permutation
$\theta: {\mathbb F}_q\rightarrow {\mathbb F}_q$
 such that the
map 
$x\rightarrow \theta(x)-x$
 is also a permutation of 
 ${\mathbb F}_q$. It is {\em canonical} if 
$\theta(0)=0$, while it is {\em quadratic} if there are two constants   $a,b\in {\mathbb F}_q\setminus\{0\}$ such that $\theta(x)=ax$, whenever $x$ is a perfect square, and $\theta(x)=bx$, otherwise. (Observe that every quadratic orthomorphism is canonical.) Furthermore, for any given orthomorphism $\theta$ over $\mathbb{F}_q$, let $L(\theta)$ denote the Latin square of order $q$ such that

\[\mathrm{Ent}(L(\theta)):=\left\{\left(i,\,j,\,\theta(j-i)+i\right)\mid i,j\in\mathbb{F}_q\right\}.\]
Then, it is possible to define an automorphism $\omega\in\mathrm{Aut}(L(\theta))$ such that $\omega(i):=i+1$, for all $i\in\mathbb{F}_q$. 

\vspace{0.2cm}

\begin{theorem}{\rm \cite{Evans1, Wanless1}} \label{EW}
Two constants $a$ and $b$ define a quadratic orthomorphism as above if and only if 
$ab$ and $(a-1)(b-1)$ are non-zero squares in 
 ${\mathbb F}_q$. 
\end{theorem}

\vspace{0.25cm}

\begin{example}\label{example_ort} The triple $(q,a,b)=(11,2,6)$ satisfies the hypothesis of Theorem \ref{EW}. Thus, the permutation

$$\theta:=(12)(36)(48)(5 \, 10)(79)$$
is a quadratic orthomorphism in $\mathbb{F}_{11}$. Then,

$$L(\theta)\equiv\begin{array}{|c|c|c|c|c|c|c|c|c|c|c|}
\hline
0 & 2 & 1 & 6 & 8 & 10 & 3 & 9 & 4 & 7 & 5  \\
\hline
6 & 1 & 3 & 2 & 7 & 9 & 0 & 4 & 10 & 5 & 8  \\
\hline
9 & 7 & 2 & 4 & 3 & 8 & 10 & 1 & 5 & 0 & 6  \\
\hline
7 & 10 & 8 & 3 & 5 & 4 & 9 & 0 & 2 & 6 & 1  \\
\hline
2 & 8 & 0 & 9 & 4 & 6 & 5 & 10 & 1 & 3 & 7  \\
\hline
8 & 3 & 9 & 1 & 10 & 5 & 7 & 6 & 0 & 2 & 4  \\
\hline
5 & 9 & 4 & 10 & 2 & 0 & 6 & 8 & 7 & 1 & 3  \\
\hline
4 & 6 & 10 & 5 & 0 & 3 & 1 & 7 & 9 & 8 & 2  \\
\hline
3 & 5 & 7 & 0 & 6 & 1 & 4 & 2 & 8 & 10 & 9  \\
\hline
10 & 4 & 6 & 8 & 1 & 7 & 2 & 5 & 3 & 9 & 0   \\
\hline
1 & 0 & 5 & 7 & 9 & 2 & 8 & 3 & 6 & 4 & 10   \\
\hline
\end{array}$$
The automorphism group $\mathrm{Aut}(L(\theta))$ is generated by the automorphisms $\alpha(x)=3x$ and $\overline{\alpha}(x)=5x-4$. Hence, $|\mathrm{Aut}(L(\theta))|=55$. Now, let $e:=(0,1,2)\in\mathrm{Ent}(L(Q))$ and let $G=\mathrm{Aut}(L(\theta))$. Conditions (C1)--(C3) in Theorem \ref{proposition_aut} hold and hence, we can construct the Latin bitrade $(T_{\tau,G},\,T'_{\tau,G})$, with $\tau:=(e,\alpha,\overline{\alpha})$. 

Since $\mathrm{Stab}_{G}(e)=\{\mathrm{Id}\}$ and $|\mathrm{Stab}^{\mathrm{row}}_{G}(e)|=|\mathrm{Stab}^{\mathrm{col}}_{G}(e)|=|\mathrm{Stab}^{\mathrm{sym}}_{G}(e)|=5$, then Lemma \ref{lemma_atop} implies that both Latin trades $T_{\tau,G}$ and $T'_{\tau,G}$ are $5$-homogeneous. They are highlighted in bold type within $L(\theta)$ as follows.

$$\begin{array}{|c|c|c|c|c|c|c|c|c|c|c|}
\hline
0 & {\bf 2_6} & 1 & {\bf 6_7} & {\bf 8_2} & {\bf 10_8} & 3 & 9 & 4 & {\bf 7_{10}} & 5  \\
\hline
6 & 1 & {\bf 3_7} & 2 & {\bf 7_8} & {\bf 9_3} & {\bf 0_9} & 4 & 10 & 5 & {\bf 8_0}  \\
\hline
{\bf 9_1} & 7 & 2 & {\bf 4_8} & 3 & {\bf 8_9} & {\bf 10_4} & {\bf 1_{10}} & 5 & 0 & 6  \\
\hline
7 & {\bf 10_2} & 8 & 3 & {\bf 5_9} & 4 & {\bf 9_{10}} & {\bf 0_5} & {\bf 2_0} & 6 & 1  \\
\hline
2 & 8 & {\bf 0_3} & 9 & 4 & {\bf 6_{10}} & 5 & {\bf 10_0} & {\bf 1_6} & {\bf 3_1} & 7  \\
\hline
8 & 3 & 9 & {\bf 1_4} & 10 & 5 & {\bf 7_{0}} & 6 & {\bf 0_1} & {\bf 2_7} & {\bf 4_2}  \\
\hline
{\bf 5_3} & 9 & 4 & 10 & {\bf 2_5} & 0 & 6 & {\bf 8_{1}} & 7 & {\bf 1_2} & {\bf 3_8}  \\
\hline
{\bf 4_9} & {\bf 6_4} & 10 & 5 & 0 & {\bf 3_6} & 1 & 7 & {\bf 9_2} & 8 & {\bf 2_3}  \\
\hline
{\bf 3_4} & {\bf 5_{10}} & {\bf 7_5} & 0 & 6 & 1 & {\bf 4_7} & 2 & 8 & {\bf 10_3} & 9  \\
\hline
10 & {\bf 4_5} & {\bf 6_0} & {\bf 8_6} & 1 & 7 & 2 & {\bf 5_8} & 3 & 9 & {\bf 0_4}   \\
\hline
{\bf 1_5} & 0 & {\bf 5_6} & {\bf 7_1} & {\bf 9_7} & 2 & 8 & 3 & {\bf 6_9} & 4 & 10   \\
\hline
\end{array}$$
\hfill $\lhd$
\end{example}

\vspace{0.25cm}

The previous example can be generalized as follows. Let $\theta$ be a quadratic orthomorphism in 
${\mathbb F}_q$
based on two constants $a$ and $b$ such that $ab$ and $(a-1)(b-1)$ are non-zero squares in 
 ${\mathbb F}_q$. In addition, let $e:=(0,1,a)\in \mathrm{Ent}(L(\theta))$ and let $G\leq\mathrm{Aut}(L(\theta))$ be subgroup that is generated by the automorphisms $\alpha(x)=bx/a$ and 
$\overline{\alpha}(x)=(b-1)(x-1)/(a-1)+1$, both of them in $\mathrm{Aut}(L(\theta))$. Conditions (C1)--(C3) in Theorem \ref{proposition_aut} are satisfied, so long as $a\neq b$. Hence, the following result holds. 


\vspace{0.1cm}

\begin{theorem}
Let $\theta$ be a quadratic orthomorphism in 
${\mathbb F}_q$
based on two distinct constants $a$ and $b$ such that $ab$ and $(a-1)(b-1)$ are non-zero squares in 
 ${\mathbb F}_q$.  
Let $m$ be the order of the group generated by $b/a$ and $(b-1)/(a-1)$ in the multiplicative group of 
${\mathbb F}_q$. 
Then, there is a Latin trade $T$ in  $L(\theta)$ of size $mq$. 
Moreover, there are 
$(q-1)/m$ disjoint trades in $L(\theta)$,  each one of them isotopic to $T$. 
\end{theorem}

\begin{proof}
From the comments in the preamble, it remains to show that there are disjoint copies of $T$.  
Let $B\leq \mathrm{Aut}(L(\theta))$ be the affine group formed by automorphisms of the form 
$f(x)=gx+h$, where 
$g$ and $h$ are respective elements of the multiplicative and additive groups of  ${\mathbb F}_q$. Then, $|B|=q(q-1)$ and $\mathrm{Stab}_B(e)=\{\mathrm{Id}\}$. Hence, the result follows from Lemma \ref{blockproof}. 
\end{proof}

\vspace{0.1cm}

\backmatter

\bmhead{Acknowledgments}

The authors are very grateful to Professor Ian Wanless, for his assistance in the development of this study. Further, Falc\'on's work is partially supported by both the Research Project {\bf FQM-016} {\em ``Codes, Design, Cryptography and Optimization''}, from Junta de Andaluc\'\i a, and the Research Project {\bf TED2021-130566B-100}  {\em ``New graphical authentication schemes in Information Management Systems by means of fractal images based on Latin squares''}, from Ministry of Science and Innovation of the Government of Spain.

\section*{Statements and Declarations}

\begin{itemize}
\item {\bf Data availability:} The datasets generated during and/or analysed during the current study are available from the corresponding author on reasonable request.

\item {\bf Competing Interests:} Authors declare no competing interest.
\end{itemize}

\end{document}